\documentclass[12pt]{article}
\RequirePackage[colorlinks,citecolor=blue,urlcolor=blue,linkcolor=blue]{hyperref}
\hypersetup{
colorlinks = true,
citecolor=blue,
urlcolor=blue,
linkcolor=blue,
pdfauthor = {R. Feng, A. Kuznetsov and F. Yang},
pdfkeywords = { hypergeometric function, basic hypergeometric function, partial fractions, non-local derangement identity},
pdftitle = {A short proof of duality relations for hypergeometric functions},
pdfpagemode = UseNone
}
\usepackage{graphicx,xspace,colortbl}
\usepackage{amsmath,amsthm,amsfonts,amssymb}
\usepackage{color}
\usepackage{enumerate}
\usepackage{fancybox}
\usepackage{epsfig}
\usepackage{subfig}
\usepackage{pdfsync}
    \def\qed{\hfill$\sqcap\kern-8.0pt\hbox{$\sqcup$}$\\}
    \def\beq{\begin{eqnarray}}
    \def\eeq{\end{eqnarray}}
    \def\beqq{\begin{eqnarray*}}
    \def\eeqq{\end{eqnarray*}}

    \def\c{{\mathbb C}} 

    \newmuskip\pFqskip
\mathchardef\pFcomma=\mathcode`,

\newtheorem{theorem}{Theorem}
\newtheorem{lemma}{Lemma}

\theoremstyle{definition}

\newtheorem{remark}{Remark}

\title{A short proof of duality relations for hypergeometric functions}
\author{ 
{Runhuan Feng\footnote{Dept. of Mathematics, University of Illinois at Urbana-Champaign, 1409 W. Green Street, 
Urbana, IL 61801, USA. Email: rfeng@illinois.edu}} ,
{Alexey Kuznetsov\footnote{Dept. of Mathematics and Statistics,  York University,
4700 Keele Street, Toronto, ON, M3J 1P3, Canada.   Email: kuznetsov@mathstat.yorku.ca}} , \;
{Fenghao Yang\footnote{Dept. of Mathematics and Statistics,  York University,
4700 Keele Street, Toronto, ON, M3J 1P3, Canada.   Email: fenghao@mathstat.yorku.ca}}
 }

\date{\today}

\begin{document}

\maketitle

\begin{abstract}
Identities involving finite sums of products of hypergeometric functions and their duals have been studied since 1930s. 
Recently Beukers and Jouhet have used an algebraic approach to derive a very general family of duality relations. In this paper we provide an alternative way of obtaining such results. Our method is very simple and it is based on the non-local derangement identity. 	 
\end{abstract}
{\vskip 0.15cm}
{\vskip 0.25cm}

\section{Introduction and main results}

Duality relations for hypergeometric functions refer to identities involving finite sums of products of two such functions. 
There is also a similar notion of duality for basic hypergeometric functions. If seems that the first instances of such formulas have appeared in 1932 in the paper \cite{Darling01} by Darling. These results have been expanded by Bailey  \cite{Bailey} in 1933, and they have been greatly generalized recently 
by Beukers and Jouhet \cite{Beukers}, who have used the theory of $D$-modules of general linear differential (or difference) equations. Our goal in this paper is to present a different approach to deriving duality relations. As we will demonstrate, our approach is elementary and it is based on the generalization of a simple fact that the sum of residues of 
a rational function is zero when the degree of the denominator is greater than one plus the degree of numerator. 

Before we present our first result, let us introduce notation and several definitions. 
We define  {\it the  hypergeometric function}  
\begin{equation}\label{def_pFr}
{}_{p}F_{r} \Big (
\begin{matrix}
b_1, \dots, b_p \\
a_1, \dots, a_r
\end{matrix} \Big \vert 
z \Big):=\sum\limits_{k\ge 0} \frac{(b_1)_k \dots (b_p)_k}{(a_1)_k \dots (a_r)_k} \times \frac{z^k}{k!}, 
\end{equation}
where $(a)_k:=\Gamma(a+k)/\Gamma(a)$ is the Pochhammer symbol. 
When $p<r+1$ it is an entire function of $z$ and when $p=r+1$ the series in \eqref{def_pFr} converges only for $|z|<1$ (though the function can be continued analytically in the cut complex plane). 

In what follows we will be working with functions represented by power series in $z$, and we will use notation $F(z) \equiv G(z)$ to mean that $F(z)=G(z)$ for all $z$ in some neighbourhood of zero. Let ${\mathcal P}_n$ be the set of polynomials of degree $n$. We say that $F(z)\equiv G(z) \; ({\textnormal{mod}} \; {\mathcal P}_n)$ if $F(z)-G(z) \in {\mathcal P}_n$. 

The following theorem is our first main result. 

\begin{theorem}\label{theorem_main_1}
Assume that $p \le r+1$,  $\{a_i\}_{1\le i \le r+1}$ are complex numbers satisfying $a_i-a_j \notin {\mathbb Z}$ for 
$1\le i < j \le r+1$, $\{b_i\}_{1\le i \le p}$ are complex numbers and 
$\{m_i\}_{1\le i \le p}$ are integers. Define $M:=\sum\limits_{i=1}^p m_i$,
\begin{equation}\label{def_coeffs_c_i}
c_i:=\frac{\prod\limits_{j=1}^p (1+a_i-b_j)_{m_j}}{ \prod\limits_{\stackrel{1\le j \le r+1}{j \neq i}} (a_i-a_j)} 
\;\;\; \textnormal{ for } \;1\le i \le r+1,
\end{equation}
and 
\begin{align}\label{def_Hz}
H(z):=\sum\limits_{i=1}^{r+1} 
c_i &\times 
{}_{p}F_{r} \Big (
\begin{matrix}
1+a_i+m_1-b_1, \dots, 1+a_i+m_p-b_p \\
1+a_i-a_1, \dots,*,\dots, 1+a_i-a_{r+1}
\end{matrix} \Big \vert 
z \Big)
\\ \nonumber
&\times 
{}_{p}F_{r} \Big (
\begin{matrix}
 b_1-a_i,b_2-a_i, \dots, b_p-a_i \\
1+a_1-a_i, \dots, *, \dots, 1+a_{r+1}-a_i
\end{matrix} \Big \vert 
(-1)^{p+r+1} z \Big),
\end{align}
where the asterisk means that the term $1+a_i-a_i$ is omitted. Assuming that  $m_i\ge 0$ for $1\le i \le r+1$, the following is true: 
\begin{itemize}
\item[(i)] If $M<r$ then $H(z)\equiv 0$;
\item[(ii)] If $M=r$ then $H(z)\equiv 1$ in the case $p\le r$, and $H(z) \equiv 1/(1-z)$ in the case $p=r+1$;
\item[(iii)] If $M=r+1$ then $H(z) \equiv C$ in the case $p\le r-1$, and $H(z) \equiv C+z$ in the case $p=r$, and 
$$
H(z)\equiv (\alpha-\beta+p) \frac{z}{(1-z)^2}+\frac{C}{1-z} \;\;\; 
{\textnormal{ in the case $p=r+1$,}} 
$$
where $\alpha=\sum\limits_{i=1}^{r+1} a_i$, $\beta=\sum\limits_{i=1}^p b_i$ 
and $C=\alpha+ \sum\limits_{i=1}^p m_i (m_i+1-2b_i)/2$.
\end{itemize}
In the case when some of $m_i$ are  negative, the above results in (i)-(iii) hold modulo ${\mathcal P}_{-\hat m}$, where  $\hat m=\min \limits_{1\le i \le p} m_i$. 
\end{theorem}

Theorem \ref{theorem_main_1} has an analogue given in terms of basic hypergeometric functions. 
We define {\it the $q$-Pochhammer symbol}
\begin{align}\label{def_qPochhammer}
(a;q)_k:=\frac{(a;q)_{\infty}}{(aq^k;q)_{\infty}}, \;\;\; a\in \c, \; |q|<1, \; k \in {\mathbb Z},  
\end{align}
where $(w;q)_{\infty}:=\prod_{j\ge 0} (1-wq^j)$. 
{\it The basic hypergeometric function} is defined as follows
\begin{equation}\label{def_r1_phi_r}
{}_{r+1} \phi_r \Big (
\begin{matrix}
 b_1, b_2, \dots, b_{r+1} \\
 a_1, a_2, \dots, a_{r}
\end{matrix} \Big \vert 
z \Big):=
\sum\limits_{k\ge 0} \frac{(b_1;q)_k (b_2;q)_k \dots (b_{r+1};q)_k}
{(a_1;q)_k (a_2;q)_k \dots (a_r;q)_k} \times \frac{z^k}{(q;q)_k}. 
\end{equation}
It is easy to see that the above series converges when $|q|<1$ and $|z|<1$. The following theorem is our second main result.

\begin{theorem}\label{theorem_main_2}
Assume that $q$ is a complex number satisfying $|q|<1$, $\{a_i\}_{1\le i \le r+1}$ are non-zero complex numbers satisfying 
$$
a_i/a_j \notin \{\dots, q^{-2},q^{-1},1,q,q^2,\dots \},
$$
$\{b_i\}_{1\le i \le r+1}$ are non-zero complex numbers and $\{m_i\}_{1\le i \le r+1}$ are integers. Define 
$M:=\sum\limits_{i=1}^{r+1} m_i$,  $M_2:=\sum\limits_{i=1}^{r+1} m_i (m_i+1)/2$ and 
\begin{equation}\label{def_ci_2}
c_i:=(-1)^M q^{-M_2} \frac{\prod\limits_{j=1}^{r+1} b_j^{m_j} (qa_i/b_j;q)_{m_j}}
{\prod\limits_{\stackrel{1\le j \le r+1}{j \neq i}} (a_i-a_j)} \;\;\;  \textnormal{ for } \; 1\le i \le r+1. 
\end{equation}
Let
\begin{align}\label{def_G}
G(z):=\sum\limits_{i=1}^{r+1} 
c_i 
&\times 
{}_{r+1}\phi_{r} \Big (
\begin{matrix}
q^{1+m_1} a_i/b_1, \dots, q^{1+m_{r+1}} a_i/b_{r+1} \\
qa_i/a_1, \dots, *,\dots, qa_i/a_{r+1}
\end{matrix} \Big \vert 
w z \Big)\\ \nonumber
& \times 
{}_{r+1}\phi_{r} \Big (
\begin{matrix}
b_1/ a_i, \dots, b_{r+1}/ a_i \\
qa_1/a_i, \dots, *,\dots, qa_{r+1}/a_i
\end{matrix} \Big \vert  z \Big),
\end{align}
where $w:=q^{-r} \prod\limits_{i=1}^{r+1} b_i/a_i$ and the asterisk means that the term $qa_i/a_i$ is omitted. Assuming that  $m_i\ge 0$ for $1\le i \le r+1$, the following is true:
\begin{itemize}
\item[(i)] If $M<r$ then $G(z)\equiv 0$;
\item[(ii)] If $M=r$ then $G(z)\equiv 1/(1-z)$;
\item[(iii)] If $M=r+1$ then 
$$
G(z)\equiv \frac{1}{1-q} \left[\frac{C}{1-z}-\frac{q\alpha-\beta}{1-qz}\right], 
$$
where $\alpha=\sum\limits_{i=1}^{r+1} a_i$,  $\beta = \sum\limits_{i=1}^{r+1} b_i$ 
and  $C=\alpha-\sum\limits_{i=1}^{r+1} b_i q^{-m_i}$.  
\end{itemize}
In the case when some of $m_i$ are  negative, the above results in (i)-(iii) hold modulo ${\mathcal P}_{-\hat m}$, where  $\hat m=\min \limits_{1\le i \le r+1} m_i$. 

\end{theorem}

\section{Proofs}

The proof of both Theorems \ref{theorem_main_1} and \ref{theorem_main_2} is based on the following lemma. This result is stated using notions of {\it a set} and {\it a multiset}. We remind the reader that the only difference in the definition of a  set $A=\{a_1,\dots,a_n\}$ and a  multiset $B=\{b_1,\dots,b_n\}$ is that all elements $a_i$ of the set must be distinct ($a_i\neq a_j$ for $i\neq j$) whereas the elements $b_i$ of a multiset may be repeated several times ($b_i$ may be equal to $b_j$ for some $i \neq j$). 

\begin{lemma}\label{Lemma1}
Assume that  $A$ is a set of $n_A$ complex numbers and  $B$ is a multiset of $n_B$ complex numbers (possibly empty). For each element $a \in A$ we define
\begin{equation}\label{def_gamma_i}
\gamma(a)=\gamma(a;A,B)=\frac{\prod\limits_{x \in B} (a -x)}{\prod\limits_{y \in A \setminus \{a\}} (a-y)}, 
\end{equation}
with the convention that the product over an empty set is equal to one. 
Then we have
\begin{align}\label{derangement_identity}
\sum_{a \in A} \gamma(a)
=
\begin{cases}
0, \;\;\; &{\textnormal{ if }} \;  n_A>n_B+1,\\
1, \;\;\; &{\textnormal{ if }} \; n_A=n_B+1, \\
\sum\limits_{a\in A} a - \sum\limits_{b \in B} b, \;\;\; &{\textnormal{ if }} \; n_A=n_B. 
\end{cases}
\end{align}
\end{lemma}
\begin{proof}
We define the rational function
\begin{equation}\label{def_rational_f}
f(z):=\frac{\prod\limits_{b \in B} (z-b)}{\prod\limits_{a\in A} (z-a)}. 
\end{equation}
Since $A$ is a set, all the numbers $a\in A$ are distinct, therefore $f(z)$ has only simple poles.
This fact and the condition  $n_A\ge n_B$ allows 
us to write the partial fraction expansion of $f(z)$ in the form
\begin{equation*}
f(z)=\delta_{n_A,n_B}+\sum\limits_{a \in A} \frac{\gamma(a)}{z-a}. 
\end{equation*}
Here $\delta_{m,n}=1$ if  $m=n$, otherwise $\delta_{m,n}=0$. 
From the above equation we obtain an asymptotic expansion of $f(z)$ as $z\to \infty$:
\begin{equation}\label{eqn_f_asymptotic_1}
f(z)=\delta_{n_A,n_B}+ z^{-1}  \sum\limits_{a \in A} \gamma(a) +O(z^{-2}). 
\end{equation}
We can obtain another asymptotic expansion of $f(z)$ if we  start from \eqref{def_rational_f}:
\begin{align}\label{eqn_f_asymptotic_2}
f(z)&=z^{n_B-n_A} \frac{\prod\limits_{b \in B} (1-bz^{-1})}
{\prod\limits_{a\in A} (1-a z^{-1})}\\
\nonumber
&=z^{n_B-n_A}+z^{n_B-n_A-1} \left[  \sum\limits_{a\in A} a - \sum\limits_{b \in B} b \right]+O(z^{n_B-n_A-2}). 
\end{align}
The desired result \eqref{derangement_identity} follows by comparing the coefficients in front of the term $z^{-1}$ in the two formulas \eqref{eqn_f_asymptotic_1} and \eqref{eqn_f_asymptotic_2}. 
\end{proof}

\begin{remark}
The result \eqref{derangement_identity} in the case $n_A=n_B$ is equivalent to {\it the nonlocal derangement identity} (see formula (1.20) in \cite{gosper1993}). In fact, the case $n_A=n_B$ is really the main one -- the other two cases can be deduced from it by a simple limiting procedure. For example, the result in the case $n_A=n_B+1$ can be deduced from the case $n_A=n_B$ as follows: take an element $b_1 \in B$, divide both sides of \eqref{derangement_identity} by $b_1$ and then let $b_1 \to \infty$. In a similar way one can derive the result in case $n_A>n_B+1$. 
\end{remark}

\vspace{0.25cm}
\noindent
{\bf Proof of Theorem \ref{theorem_main_1}:} 
Let us prove the first part of Theorem \ref{theorem_main_1}: we assume that $m_i \ge 0$ for $1\le i \le p$.
Let $k$ be a non-negative integer. We define 
\begin{equation}\label{def_set_A}
A=\bigcup_{1\le i \le r+1} \{ a_i + j \; : \; 0\le j \le k\}. 
\end{equation}
Note that the condition $a_i - a_j \notin {\mathbb Z}$ for $1\le i < j <r+1$ implies that the set $A$ has 
$n_A=(r+1)(k+1)$ elements. Similarly, we define a multiset
\begin{equation}\label{def_set_B}
B=\biguplus_{1\le i \le p} \{ b_i+j  \; : \; -m_i\le j \le k-1\}. 
\end{equation}
The symbol ``$\biguplus$" means that we are taking union of multisets; in other words, one complex number may be repeated 
several times in $B$. It is clear that the multiset $B$ has $n_B= M+kp$ elements (recall that $M=m_1+\dots+m_p$). 


Let us fix $i$ and $j$ such that  $1\le i \le r+1$ and $0\le j \le k$ and consider the element $a_i+j$ of the set $A$. 
From formula \eqref{def_gamma_i} we find 
\begin{align}\label{eqn_gamma_ij}
\nonumber
\gamma^k_{i,j}:=\gamma(a_i+j; A,B)&=\frac{\prod\limits_{x \in B} (a_i+j-x)}
{\prod\limits_{y \in A\setminus \{a_i+j\}} (a_i+j-y)}\\ 
&= \frac{\prod\limits_{l=1}^p \prod\limits_{s=-m_l}^{k-1} (a_i+j-b_l-s)}
{\prod\limits_{\stackrel{0\le s \le k}{s \neq j}}(j-s) \prod\limits_{\stackrel{1\le l \le r+1}{l \neq i}} \prod\limits_{s=0}^{k} (a_i+j-a_l-s)
}
\end{align}
Now we will simplify the expression in \eqref{eqn_gamma_ij}. We check that 
$$
\prod\limits_{\stackrel{0\le s \le k}{s \neq j}}(j-s)=(-1)^{k-j} j! (k-j)!
$$
and for any $w \in \c$, $m\ge 0$, $k\ge 0$   and $0\le j \le k$
\begin{equation}\label{eqn_Pochhammer_identity1}
\prod\limits_{s=-m}^{k-1} (w+j-s)=(-1)^{k-j} (1+w)_{m} (1+m+w)_j (-w)_{k-j}. 
\end{equation}
The above two identities allow us to rewrite the expression in \eqref{eqn_gamma_ij} as follows
\begin{align}\label{formula_gamma_ikj}
\gamma^k_{i,j}=\frac{\prod\limits_{l=1}^p (1+a_i-b_l)_{m_l}}{ \prod\limits_{\stackrel{1\le l \le r+1}{l \neq i}} (a_i-a_l)}
\times 
\frac{1}{j!} \times 
\frac{\prod\limits_{l=1}^p (1+m_l+a_i-b_l)_{j}}{\prod\limits_{\stackrel{1\le l \le r+1}{l \neq i}} (1+a_i-a_l)_j}
\times 
\frac{(-1)^{(k-j)(p+r+1)}}{(k-j)!}
\frac{\prod\limits_{l=1}^p (b_l-a_i)_{k-j}}{\prod\limits_{\stackrel{1\le l \le r+1}{l \neq i}} (1-a_i+a_l)_{k-j}}.
\end{align}
Using the above equation and formulas  \eqref{def_pFr}, \eqref{def_coeffs_c_i} and \eqref{def_Hz} we see that 
\begin{equation}\label{Hz_in_terms_of_gamma}
 \sum\limits_{i=1}^{r+1} \sum\limits_{k\ge 0} z^k \sum\limits_{j=0}^k 
\gamma_{i,j}^k =H(z). 
\end{equation}

At the same time, we can  change the order of summation in \eqref{Hz_in_terms_of_gamma}
and write $H(z)$ as 
\begin{equation}\label{Hz_in_terms_of_gamma2}
H(z)= \sum\limits_{k\ge 0} z^k \left[ \sum\limits_{i=1}^{r+1}\sum\limits_{j=0}^k \gamma_{i,j}^k \right]. 
\end{equation}
Now the plan is to compute the sum in the square brackets by applying Lemma \ref{Lemma1}. Recall that we have denoted $\alpha=\sum\limits_{i=1}^{r+1} a_i$ and
$\beta=\sum\limits_{i=1}^{p} b_i$. Definitions 
\eqref{def_set_A} and \eqref{def_set_B} easily give us 
$$
s_k:=\sum\limits_{x\in A} x - \sum\limits_{y\in B} y=(k+1) \alpha+(r+1) \frac{k(k+1)}{2}- k \beta 
- p \frac{(k-1)k}{2}+\frac{1}{2}\sum\limits_{i=1}^p m_i (m_i+1-2b_i). 
$$
Then, using our earlier computations $n_A=(r+1)(k+1)$ and $n_B= M+kp$ and applying Lemma \ref{Lemma1}, we find 
\begin{align}\label{derangement_identity2}
\sum_{i=1}^{r+1} \sum\limits_{j=0}^k \gamma_{i,j}^k
=
\begin{cases}
0, \;\;\; &{\textnormal{ if }} \;  (r+1-p)k>M-r,\\
1, \;\;\; &{\textnormal{ if }} \;  (r+1-p)k=M-r, \\
s_k, \;\;\; &{\textnormal{ if }} \; (r+1-p)k=M-r-1.
\end{cases}
\end{align}
By combining \eqref{Hz_in_terms_of_gamma2} and  \eqref{derangement_identity2} we finish the proof
of Theorem \ref{theorem_main_1} in the case when $m_i \ge 0$ for $1\le i \le p$.

Let us consider the case when some $m_i$ are negative. Note that formula \eqref{eqn_Pochhammer_identity1} holds true when $m$ is negative, as long as $k\ge |m|$. Thus formula  \eqref{formula_gamma_ikj} is also true, as long as 
$k \ge |m_i|$ for all negative $m_i$. Therefore, our result \eqref{derangement_identity2} remains true for 
all $k \ge - \hat m$ (recall that $\hat m =\min\{m_i \; : \; 1\le i \le p\}$), which means that all results in Theorem 
\ref{theorem_main_1} hold true modulo ${\mathcal P}_{-\hat m}$. 
\qed

\vspace{0.25cm}
\noindent
{\bf Proof of Theorem \ref{theorem_main_2}:}
The proof is very similar to the proof of Theorem \ref{theorem_main_1}, thus we will present only the important steps and 
we will omit many details. 
Assume that $m_i \ge 0$ for $1\le i \le r+1$ and $k\ge 0$ (or $k \ge -\hat m$ if some of $m_i$ are negative). 
We define 
\begin{align*}
A=\bigcup\limits_{1\le i \le r+1}  \{ a_i q^j \; : \; 0 \le j \le k \}, \qquad B=\biguplus\limits_{1\le i \le r+1} \{ b_i q^j \; : \; -m_i \le j \le k-1\}.
\end{align*}
It is clear that $n_A=(r+1)(k+1)$ and $n_B=M+(r+1)k$. 
Next, we fix indices $i$ and $j$ such that $1\le i \le r+1$ and $0\le j \le k$ and compute 
\begin{equation}\label{eqn_gamma_ij2}
\gamma_{i,j}^k:=\gamma(a_i q^j; A, B)=\frac{\prod\limits_{l=1}^{r+1} \prod\limits_{s=-m_l}^{k-1} (a_iq^{j}-b_lq^{s})}
{\prod\limits_{\stackrel{0\le s \le k}{s \neq j}}(a_i q^{j}-a_i q^{s}) 
\prod\limits_{\stackrel{1\le l \le r+1}{l \neq i}} \prod\limits_{s=0}^{k} (a_i q^{j}-a_lq^{s}). 
}
\end{equation}
After some straighforward (though tedious) computations we rewrite the above expression in the form
\begin{align}\label{formula_gamma_ijk2}
\gamma_{i,j}^k=(-1)^M q^{- M_2} \frac{\prod\limits_{l=1}^{r+1} b_l^{m_l} (qa_i/b_l;q)_{m_l}}
{\prod\limits_{\stackrel{1\le l \le r+1}{l \neq i}} (a_i-a_l)}
&\times 
\frac{ w^{j}  \prod\limits_{l=1}^{r+1} (q^{1+m_l} a_i/b_l;q)_{j}}
{(q;q)_j \prod\limits_{\stackrel{1\le l \le r+1}{l \neq i} }(q a_i/a_l;q)_j } \\ \nonumber
&\times 
\frac{  \prod\limits_{l=1}^{r+1} (b_l/ a_i;q)_{k-j}}
{(q;q)_{k-j} \prod\limits_{\stackrel{1\le l \le r+1}{l \neq i}} (q a_l/a_i;q)_{k-j} },
\end{align}
which shows that
\begin{equation}\label{eqn_Gz_gamma}
 \sum\limits_{i=1}^{r+1} \sum\limits_{k\ge 0} z^k \sum\limits_{j=0}^k 
\gamma_{i,j}^k =G(z),
\end{equation}
where the function $G(z)$ is defined in \eqref{def_G}. 
We also compute 
$$
s_k:=\sum\limits_{x\in A} x - \sum\limits_{y\in B} y
=\frac{1}{1-q} \Big[ \alpha- \sum\limits_{l=1}^{r+1} b_l q^{-m_l}
- (q \alpha -  \beta)q^k \Big],
$$
and Lemma \ref{Lemma1} gives us
\begin{align}\label{derangement_identity3}
\sum_{i=1}^{r+1} \sum\limits_{j=0}^k \gamma^k_{i,j}
=
\begin{cases}
0, \;\;\; &{\textnormal{ if }} \;    M<r,\\
1, \;\;\; &{\textnormal{ if }} \;    M=r, \\
s_k, \;\;\; &{\textnormal{ if }} \; M=r+1.
\end{cases}
\end{align} 
The remaining steps of the proof are exactly the same as in the proof of Theorem \ref{theorem_main_1} and we leave them to the reader. 
\qed

\section*{Acknowledgements}

 The research of A. Kuznetsov is supported by the Natural Sciences and Engineering Research Council of Canada. 
  We would like to thank Mourad Ismail for suggesting the non-local derangment identity (Lemma \ref{Lemma1}) as a simpler way of proving Theorem \ref{theorem_main_1} -- our original proof was more complicated and it was based on Meijer G-function. We are also grateful to two referees, who have carefully read the earlier version of this manuscript and have provided several valuable comments and have pointed out relevant references.

%

\end{document}